\documentclass[a4paper,10pt]{article}

\usepackage{bm}
\usepackage{amsmath}
\usepackage{amssymb}
\usepackage{amsthm}
\usepackage{mathrsfs}
\usepackage{color}
\usepackage[dvips]{graphicx}
\usepackage[all]{xy}
\usepackage{array, booktabs}

\theoremstyle{plain}
\newtheorem{them}{Theorem}[section]
\newtheorem{lemma}[them]{Lemma}
\newtheorem{prop}[them]{Proposition}

\theoremstyle{definition}
\newtheorem{defi}[them]{Definition}
\newtheorem{exam}[them]{Example}
\newtheorem{rema}[them]{Remark}

\newtheorem{conj}[them]{Conjecture}

\newcommand{\Li}{\mathop{\mathrm{Li}}}

\begin{document}

\title{Zeros of the zeta series of a poset and iterated barycentric subdivision}
\author{Kazunori Noguchi \thanks{noguchi@cc.kogakuin.ac.jp}}
\date{}
\maketitle
\begin{abstract}
We study the limiting behavior of the zeros of the zeta series of a finite poset under iterated barycentric subdivision, and we indicate the possibility of its application to number theory.
\end{abstract}


\thispagestyle{empty}

\section{Introduction}

Zeros of a holomorphic function is very important for the function itself and its applications. If the function is  a polynomial $f(z)$, we use solutions of  the equation $f(z)=0$ in all area of mathematics. If the function is the Riemann zeta function $\zeta(s)$, the location of its complex zeros are very important. The Riemann hypothesis states that such zeros lie on the line $\Re s=\frac{1}{2}$, and the conjecture implies a very precise result of the distribution of primes. 

In this paper, we study the limiting behavior of the zeros of the zeta series of a finite poset under iterated barycentric subdivision, and we give a plan, not accomplished, for application of our main result to the distribution of primes.

Let $C$ be a finite category. Define the \textit{zeta series} $Z_C(s)$ of $C$ by $$ Z_C(s):= \sum^{\infty}_{i=0} \# N_i(C) s^i,$$
where $N_i(C)$ is the set of chains of morphisms of length $i$ and $s$ is a complex number. We are allowed to use identity morphisms in the chains. The function $Z_C(s)$ is almost the logarithmic derivative of the zeta function $\zeta_C(s)$ of $C$ \cite{NogB}; that is, $$Z_C(s)= \# N_0(C)+s \frac{\zeta_C'(s)}{\zeta_C(s)},$$ where $$\zeta_C(s)=\exp\left( \sum_{i=1}^{\infty} \frac{\# N_i(C) }{i} s^i \right).$$ The remarkable property of $Z_C(s)$ is to recover the Euler characteristic of $C$, in the sense of \cite{BL08}, by residues (Corollary 3.6 of \cite{NogB}): $$\chi(C)=\mathrm{Res} (Z_C(s) : \infty),$$
if $\chi(C)$ exists. 

Barycentric subdivision is a familiar notion in topology, and it is also defined for posets. It is an operation, denoted by $\mathrm{Sd}$, to produce a new poset from a poset. See the next section for more detail.

The following is  our main theorem.

\begin{them}\label{main}
Suppose that $P$ is a finite poset of dimension $ d  \ge 1$ and its Euler characteristic is nonzero. Let $\beta_1^{(k)},\dots,\beta_d^{(k)}$ be the zeros of $Z_{ \mathrm{Sd}^k(P)  } (s)$ and we assume $| \beta_1^{(k)} |  =\max\{    | \beta_1^{(k)} | ,  \dots, | \beta_d^{(k)} |    \}$. Then, as $k\to \infty$, $| \beta_1^{(k)} | $ diverges to $\infty$ and the others $\beta_i^{(k)}$ converge. In particular, $\beta_1^{(k)}$ is real for sufficiently large $k$ and the product $\prod^d_{i=2} \beta_i^{(k)}$ converges to $(-1)^{d-1}$.
\end{them}

The Euler characteristic of a finite poset $P$ is given by $$\chi(P) = \sum^{\dim P}_{i=0} (-1)^i \# \overline{N_i} (P),$$ where $\overline{N_i} (P)$ is the set of chains of nonidentity morphisms in $P$ of length $i$, and the dimension of $P$ is the greatest integer $i$ such that $\overline{N_i} (P)$ is not empty.

Note that this result is very similar to Theorem 3 of \cite{BW08} and Theorem A of \cite{DPS12}, but the convergence of the product is stronger than theirs.

In the last section, we introduce a plan for application of the main theorem to the distribution of primes. It gives a continuation of Bj\"orner's topological approch to the difficult problem \cite{Bjo11}.

\section{Main theorem}

Throughout this section, $P$ is a finite poset.

\subsection{Preliminaries}

We first show the rationality of the zeta series of a finite category $C$. Suppose that the set of objects of $C$ is $\{ x_1,\dots, x_k\}$. Define the \textit{adjacency matrix} $A_C$ to be the matrix whose $(i,j)$-entry is the number of morphisms from $x_i$ to $x_j$.

\begin{lemma}\label{rational} For a finite category $C$, the zeta series of $C$ is rational; that is, $$ Z_C(s)=\frac{  \mathrm{sum}\{  \mathrm{adj} (I-A_C s) \}  }{ \det (I-A_C  s)    } ,$$
where sum means to take the sum of all the entries of a matrix.
\begin{proof}
We have $\# N_i(C) =\mathrm{sum}(A_C^i)$ for any $i\ge 0$. Hence, the result follows from Lemma 2.1 of \cite{BL08}.
\end{proof}

\end{lemma}

\begin{lemma}\label{zero}

A complex number $s_0$ is a zero of $Z_P(s)$ if and only if $s_0$ is that of $$ \sum^{\dim P}_{i=0} \#\overline{N_i}(P) s^i (1-s)^{\dim P -i}.$$
\begin{proof}
Since $$ Z_P(s)=\# N_0(P) + s \frac{\zeta_P'(s)}{\zeta_P(s)},$$
by Corollary 2.12 of \cite{NogB}, we have
$$ Z_P(s) = \sum^{\dim P}_{i=0} \frac{ \#\overline{N_i} (P) s^i }{ (1-s)^{i+1} } =\frac{  \sum^{\dim P}_{i=0} \#\overline{N_i} (P) s^i(1-s)^{\dim P -i} }{ (1-s)^{\dim P+1}    }.$$
Hence, $s=1$ is a unique pole of $Z_P(s)$ and the numerator of the right hand side is $\# \overline{N_{\dim P}} (P)$ when $s=1$. Since  $\# \overline{N_{\dim P}} (P)$ is nonzero, the pole does not vanish the zeros of the numerator. Hence, the result follows.
\end{proof}
\end{lemma}

We write $$  g_P(s):=  \sum^{\dim P}_{i=0} \#\overline{N_i} (P) s^i(1-s)^{\dim P -i} . $$

Next, we define the \textit{barycentric subdivision} Sd$(P)$ of $P$ as follows: the set of objects of Sd$(P)$ is the coproduct $\coprod_{i\ge 0} \overline{N_i} (P)$, and its order is given by inclusion. Here, we regard a chain $x_0\to x_1\to\cdots\to x_i $ of  $\overline{N_i} (P)$ as the totally ordered $(i+1)$-subset $\{ x_0,x_1,\dots,x_i \}$ of $P.$

For example, if $P=x\longrightarrow y$, then Sd$(P)$ is $$ \{x\} \longrightarrow\{x,y\}\longleftarrow \{ y\}.$$

In fact, we can define the barycentric subdivision for small categories, but the restriction is adequate for our purpose. See \cite{Nog11}.

Barycentric subdivision preserves Euler characteristic: $$\chi(P) = \chi( \mathrm{Sd} (P))$$
by Proposition 3.11 of \cite{NogA}. We denote the $k$-times subdivided poset by $P^{(k)}$:
$$ \mathrm{Sd}^{k} (P)=P^{(k)}.$$

\subsection{Various numbers and polynomials}\label{vari}

We introduce various numbers and polynomials. They are important to study combinatorial properties of barycentric subdivision.

For $i,d \ge 0$ and $\bm{f}$ of $\overline{N_d}(P)$, define $f_{i,d}$ to be the number of chains of nonidentity morphisms in $P^{(1)}$ of length $i$ whose target is $\bm{f}$. This definition does not depend on the choice of $\bm{f}$, only does on the length of $\bm{f}$. Define $f_{-1,-1}=1$ and $f_{-1, d} =0.$ It is easy to show that
\begin{align}
 f_{i,d} &= \sum^d_{j=i} \binom{d+1}{j}  f_{i-1,j-1} \label{ind1}
 \end{align}
for any $i,d\ge 0$; therefore, we can compute the numbers inductively.
\begin{center}
\begin{tabular}{lrrrrrrrr} \toprule
$f_{i,d}$&$d=0$&$d=1$&$d=2$&$d=3$&$d=4$&$d=5$&$d=6$&$d=7$ \\ \midrule
$i=0$ &1&1&1&1&1&1&1&1 \\
$i=1$&0&2&6&14&30&62&126&254 \\
$i=2$&0&0&6&36&150&540&1806&5796 \\
$i=3$&0&0&0&24&240&1560&8400&40824 \\
$i=4$&0&0&0&0&120&1800&16800&126000 \\
$i=5$&0&0&0&0&0&720&15120&191520 \\
$i=6$& 0&0&0&0&0&0&5040&141120 \\
$i=7$ & 0&0&0&0&0&0&0&40320 \\ \bottomrule
\end{tabular}
\end{center}
In particular, we have $f_{d,d} =(d+1)!$, $f_{d-1,d} =\frac{d}{2} (d+1)!$, and $f_{1,d} =2(2^d -1)$.

We define the next number. For $0\le i<d$, define a rational number $F_{i,d}$ by

\begin{multline}F_{i,d} = \sum_{  i_0=i<i_1<\cdots <i_{\ell}<d, \ell \ge 0 }   \frac{  f_{i_0,i_1} }{ (d+1)!-(i_0 +1)!  } \frac{  f_{i_1,i_2} }{ (d+1)!-(i_1 +1)!  } \cdots \\ \cdots\frac{  f_{i_{\ell}, d} }{ (d+1)!-(i_{\ell} +1)!  }  .\end{multline}
Put $F_{d,d}=1$ and $F_{-1,d}=0$ if $d\ge 0$. At first glance, the numbers seem to be complicated, so some readers might wonder what the numbers are. However, we will see that the column vector ${}^t (F_{-1,d},  F_{0,d} ,\dots,F_{d,d})$ is an eigenvector of a certain matrix (Lemma \ref{sym1}). We can inductively compute $F_{i,d}$ by the following: 
 \begin{align}F_{i,d} &=\frac{1}{  (d+1)!-(i+1)!  }  \sum^d_{j=i+1}  f_{i,j} F_{j,d} \label{ind2}\end{align}
 for $-1\le i\le d-1$.  

\begin{center}
\begin{tabular}{ccccccccc} \toprule
$F_{i,d}$ & $d=0$& $d=1$& $d=2$& $d=3$& $d=4$& $d=5$& $d=6$& $d=7$ \\ \midrule  \addlinespace[3pt]
$i=0$&1&1&$\frac{1}{2}$ & $\frac{2}{11}$&$\frac{1}{19}$& $\frac{132}{10411} $& $\frac{90}{34399}$ & $\frac{15984}{33846961}$\\  \addlinespace[3pt]
$i=1$&&1&$\frac{3}{2} $&$\frac{13}{11}$ & $\frac{25}{38}$&$\frac{3004}{10411}$& $\frac{3626}{34399}$& $\frac{12351860}{372316571}$\\  \addlinespace[3pt]
$i=2$&&&1&2&$\frac{40}{19}$& $\frac{45}{29}$& $\frac{61607}{68798}$&$\frac{7924}{18469}$ \\  \addlinespace[3pt]
$i=3$&&&&1&$\frac{5}{2}$&$\frac{95}{29}$&$\frac{245}{82}$& $\frac{39221}{18469}$ \\  \addlinespace[3pt]
$i=4$&&&&&1&3&$\frac{385}{82}$& $\frac{56}{11}$ \\  \addlinespace[3pt]
$i=5$&&&&&&1&$\frac{7}{2}$& $\frac{70}{11}$ \\  \addlinespace[3pt]
$i=6$&&&&&&&1&4 \\  \addlinespace[3pt]
$i=7$&&&&&&&&1 \\ \bottomrule
\end{tabular}
\end{center}
We can find the same table, but bigger than ours, in \S 6  of \cite{DPS12}.

 Define the \textit{$F$-polynomial} $F_d(s)$ of degree $d\ge 0$ by $$ F_d(s) :=\sum^d_{i=-1} F_{i,d} s^{d-i}.$$

Finally, we define the most important number and polynomial in this paper. Define a rational number $H_{i,d}$, for $ 0\le i \le d,$ to be the coefficient of $(s+1)^i$ in the Taylor expansion of $F_d$(s) at $s=-1$: $$ F_d(s)=\sum^{d+1}_{i=0}  H_{i,d} (s+1)^i.$$ In other words, we have $$ F_d(s-1)=\sum^{d+1}_{i=0} H_{i,d} s^i.$$

\begin{defi}
For $d \ge0$, define the \textit{$H$-polynomial} $H_d(s)$ of degree $d$ by $$ H_d(s):= \sum^{d+1}_{i=0} H_{i,d} s^{d-i}.$$
\end{defi}

This polynomial is the heart of the proof of the main theorem. We will find that zeros of $H_{\dim P}(s)$ directly influences to those of $Z_{P^{(k)}} (s)$. By the table above, we obtain a new table.
\begin{center}
\begin{tabular}{ccccccccc} \toprule
$H_{i,d}$ & $d=0$& $d=1$& $d=2$& $d=3$& $d=4$& $d=5$& $d=6$& $d=7$ \\ \midrule \addlinespace[3pt]
$i=0$&0&0&0&0&0&0&0&0 \\ \addlinespace[3pt]
$i=1$&0&1&$\frac{1}{2}$&$\frac{2}{11}$&$\frac{1}{19}$&$\frac{132}{10411}$&$\frac{90}{34399}$& $\frac{15984}{33846961}$ \\ \addlinespace[3pt]
$i=2$&&0&$\frac{1}{2}$&$\frac{7}{11}$&$\frac{17}{38}$& $\frac{2344}{10411}$&$\frac{3086}{34399}$&$\frac{11121092}{372316571}$\\ \addlinespace[3pt]
$i=3$&&&0&$\frac{2}{11}$&$\frac{17}{38}$&$\frac{5459}{10411}$&$\frac{28047}{68798}$& $\frac{89321060}{372316571}$ \\ \addlinespace[3pt]
$i=4$&&&&0&$\frac{1}{19}$& $\frac{2344}{10411}$& $\frac{28047}{68798}$& $\frac{171080619}{372316571}$ \\ \addlinespace[3pt]
$i=5$&&&&&0&$\frac{132}{10411}$& $\frac{3086}{34399}$& $\frac{89321060}{372316571}$ \\ \addlinespace[3pt]
$i=6$&&&&&&0&$\frac{90}{34399}$&$\frac{11121092}{372316571}$ \\ \addlinespace[3pt]
$i=7$&&&&&&&0&$\frac{15984}{33846961}$\\\addlinespace[3pt]
$i=8$&&&&&&&&0 \\ \bottomrule

\end{tabular}
\end{center}
We can observe that any column in the table is symmetric; therefore, $H_d(s)$ is self-reciprocal. We prove it, for any $d \ge 0$, in the next section, and the fact plays a crucial role for the proof of our main theorem.

\subsection{Symmetry of the $H$-polynomials}

In this section, we prove that the $H$-polynomials are self-reciprocal. 

For $d\ge 0$, define a matrix $\bm{F}_d$ by $\bm{F}_d =(f_{i,j})_{-1\le i,j \le d}$. For example, we have
$$ \bm{F}_0 =\begin{pmatrix}  1&0 \\0&1\end{pmatrix}   , \bm{F}_1 = \begin{pmatrix}  1&0&0  \\ 0&1  &1 \\0&0&2\end{pmatrix}  , \bm{F}_2 = \begin{pmatrix}  1&0&0 &0 \\ 0&1  &1&1 \\0&0&2 &6\\ 0&0&0&6\end{pmatrix}   .$$

\begin{lemma}\label{sym1}
For $d\ge 0$, the numbers $0!,1!,\dots, (d+1)!$ are the eigenvalues of $\bm{F}_d$, and the column vector $$ {}^t(F_{-1,d}, F_{0,d},\dots,F_{i,d},\dots,F_{d,d})$$ is an eigenvector for $(d+1)!$ of $\bm{F}_d$.
\begin{proof}
Since $\bm{F}_d$ is upper triangular and the numbers $0!,1!,\dots,(d+1)!$ are the diagonal entries, the first claim follows.

For the second claim, we have to show $$ \sum^d_{j=i} f_{i,j} F_{j,d} =(d+1)! F_{i,d}$$
for any $-1\le i \le d$. When $i=d$, it is clear. If $-1\le i\le d-1$, then \eqref{ind2} directly implies the equality. Hence, the result follows. 
\end{proof}
\end{lemma}

For $d\ge 1$ and a permutation $\sigma$ on the set $[d]=\{ 1,2,\dots,d \}$, define des$(\sigma)$ to be the number of $1\le i\le d-1$ such that $\sigma(i)>\sigma(i+1)$. For $1\le j\le d$ and $0\le i \le d-1$, we denote by $A(d,i,j)$ the number of permutations $\sigma$ on $[d]$ such that $\sigma(1)=j$ and des$(\sigma)=i$. In particular, $A(d,i,j)=0$ if $i\le -1$.

For $d\ge 0$, define a matrix $\bm{H}_d$ by $$ \bm{H}_d =\left( h^{(d)}_{i,j} \right)_{-1\ge i,j \ge d}= \big( A(d+2,i+1,j+2) \big)_{-1\le i,j \le d}.$$For example, we have $$ \bm{H}_0 =\begin{pmatrix}  1&0 \\0&1\end{pmatrix}   , \bm{H}_1 = \begin{pmatrix}  1&0&0  \\ 1&2  &1 \\0&0&1 \end{pmatrix}  , \bm{H}_2 = \begin{pmatrix}  1&0&0 &0 \\ 4&4 &2&1 \\ 1&2&4 &4 \\ 0&0&0&1 \end{pmatrix} , $$
For $d\ge 0$, define a matrix $\bm{H}_d$ by $$ \bm{H}_d =\left( h^{(d)}_{i,j} \right)_{-1\ge i,j \ge d}= \big( A(d+2,i+1,j+2) \big)_{-1\le i,j \le d}.$$For example, we have $$ \bm{H}_0 =\begin{pmatrix}  1&0 \\0&1\end{pmatrix}   , \bm{H}_1 = \begin{pmatrix}  1&0&0  \\ 1&2  &1 \\0&0&1 \end{pmatrix}  , \bm{H}_2 = \begin{pmatrix}  1&0&0 &0 \\ 4&4 &2&1 \\ 1&2&4 &4 \\ 0&0&0&1 \end{pmatrix} , $$
$$  \bm{H}_3 = \begin{pmatrix}  1&0&0&0&0 \\ 11&8&4&2&1 \\ 11&14&16&14&11 \\1&2&4&8&11 \\0&0&0&0&1 \end{pmatrix}  , \bm{H}_4=\begin{pmatrix} 1&0&0&0&0&0 \\26&16&8&4&2&1 \\66&66&60&48&36&26 \\26&36&48&60&66&66 \\ 1&2&4&8&16&26 \\0&0&0&0&0&1 \end{pmatrix} .$$ 
It is easy to compute these examples by the following lemma:
\begin{lemma}[Lemma 2(i) of \cite{BW08}]\label{BW}
For $d \ge 0$ and $-1 \le i,j \le d,$ we have $$ h^{(d)}_{i,j}  = \sum^{j-1}_{\ell =-1}   h^{(d-1)}_{i-1,\ell}   + \sum^{d-1}_{\ell =j}   h^{(d-1)}_{i,\ell} .$$
\end{lemma}

The two matrices $\bm{F}_d$ and $\bm{H}_d$ have already been used in \cite{BW08}. In the paper, Brenti and Welker found that the matrices are similar (Lemma 4 (i) of \cite{BW08}); that is, there exists a nonsingular matrix $\bm{P}$ such that $\bm{P}^{-1} \bm{F}_d \bm{P} =\bm{H}_d$. In this paper, we explicitly describe the matrices $\bm{P}$ and $\bm{P}^{-1}$.

For $d\ge 0$, define the \textit{Taylor expansion matrix (at $s=-1$)} $\bm{T}_d$ by $$ \bm{T}_d =\left(  (-1)^{d+1+i+j}\binom{d-j}{i+1} \right)_{-1\le i,j \le d} .$$ If $f(s)=\sum^d_{j=-1} a_j s^{d-j}$, then we have $$ \bm{T}_d{}^t(a_{-1},a_0,\dots,a_d) ={}^t \left(   f(-1), f'(-1) ,\dots, \frac{f^{(d+1)} (-1)}{(d+1)!}\right).$$

\begin{lemma}\label{sym2}
For $d\ge 0$, the Taylor expansion matrix is nonsingular, and the inverse matrix is given by $$\bm{T}_d'= \left(  \binom{j+1}{d-i} \right)_{-1\le i,j \le d} .$$
\begin{proof}
The $(i,j)$-entry of $\bm{T}_d \bm{T}_d'$ is $$ \sum_{-1\le k\le d} (-1)^{d+1+i+k}  \binom{d-k}{i+1} \binom{j+1}{d-k},$$
and we show that it is the Kronecker delta. By multiplying by $x^{i+1}$ and summing over $k\ge -1$, we have
\begin{align*}
&\sum^{\infty}_{i=-1} \sum_{-1 \le k\le d} (-1)^{d+1+i+k} \binom{d-k}{i+1} \binom{j+1}{d-k} x^{i+1}  \\
=& \sum^{\infty}_{k=-1}(-1)^{d+k} \binom{j+1}{d-k} \sum^{\infty}_{i=0}(-1)^i \binom{d-k}{i}x^i \\
=&\sum^{\infty}_{k=-1}(-1)^{d-k}\binom{j+1}{d-k} (1-x)^{d-k} \\
=&x^{j+1}.
\end{align*}
Hence, the result follows.
\end{proof}

\end{lemma}

\begin{lemma}\label{sym3}
For $d\ge 0$, we have $\bm{T}_d \bm{F}_d \bm{T}_d^{-1}  =\bm{H}_d$.
\begin{proof}
We prove the claim by induction on $d$. 

When $d=0$, it is clear.

Suppose that the equality holds for $d-1$. The $(i,j)$-entries of both sides are $$ \sum_{-1\le k,k'\le d}  (-1)^{d+1+i+k} \binom{d-k}{i+1} \binom{j+1}{d-k'}  f_{k,k'}$$
and $h^{(d)}_{i,j}$, respectively. By Lemma \ref{BW} and the assumption of induction, we have
\begin{align*}
h^{(d)}_{i,j} =&\sum^{j+1}_{\ell=1} A(d+1,i,\ell) + \sum^{d+2}_{\ell=j+3}A(d+1,i+1,\ell -1)  \\
=&\sum_{-1\le k,k'\le d-1} (-1)^{d+1+i+k} f_{k,k'}\Bigg( \sum^j_{\ell=0} \binom{d-1-k}{i} \binom{\ell}{d-1-k'} \\
& - \sum^d_{\ell=j+1} \binom{d-1-k}{i+1}\binom{\ell}{d-1-k'} \Bigg)\\
=&\sum_{-1\le k,k'\le d-1} (-1)^{d+1+i+k} f_{k,k'}\Bigg( \sum^j_{\ell=d-1-k'} \binom{d-1-k}{i} \binom{\ell}{d-1-k'} \\
&+  \sum^j_{\ell=d-1-k'} \binom{d-1-k}{i+1} \binom{\ell}{d-1-k'} - \sum^d_{\ell=d-1-k'} \binom{d-1-k}{i+1}\binom{\ell}{d-1-k'} \Bigg)\\
=&\sum_{-1\le k,k'\le d-1} (-1)^{d+1+i+k} f_{k,k'}\left(  \binom{d-k}{i+1}\binom{j+1}{d-k'} -\binom{d-1-k}{i+1}\binom{d+1}{k'+1} \right) .
\end{align*}
The last equality follows from the equality $\sum^b_{c=a}\binom{c}{a}=\binom{b+1}{a+1}$ for $a,b \ge 0$.

On the other hand, by \eqref{ind1}, we have
\begin{align*}
&\sum_{-1\le k,k' \le d} (-1)^{d+1+i+k} \binom{d-k}{i+1} \binom{j+1}{d-k'}f_{k,k'}\\
=&\sum_{-1\le k,k' \le d-1} (-1)^{d+1+i+k} \binom{d-k}{i+1} \binom{j+1}{d-k'}f_{k,k'} \\
&+\sum_{0\le k \le d} (-1)^{d+1+i+k} \binom{d-k}{i+1}  f_{k,d}\\
=&\sum_{-1\le k,k' \le d-1} (-1)^{d+1+i+k} \binom{d-k}{i+1} \binom{j+1}{d-k'}f_{k,k'} \\
& +\sum_{0\le k \le d} (-1)^{d+1+i+k} \binom{d-k}{i+1} \sum^d_{k'=k} \binom{d+1}{k'} f_{k-1,k'-1} \\
=&\sum_{-1\le k,k' \le d-1} (-1)^{d+1+i+k} \binom{d-k}{i+1} \binom{j+1}{d-k'}f_{k,k'} \\
&+ \sum_{-1\le k,k' \le d-1} (-1)^{d+i+k} \binom{d-1-k}{i+1} \binom{d+1}{k'+1}f_{k,k'} \\
=&\sum_{-1\le k,k'\le d-1} (-1)^{d+1+i+k} f_{k,k'}\left(  \binom{d-k}{i+1}\binom{j+1}{d-k'} -\binom{d-1-k}{i+1}\binom{d+1}{k'+1} \right) .
\end{align*}
Hence, the result follows.
\end{proof}
\end{lemma}

A square matrix $\bm{A}=(a_{i,j})_{1\le i,j \le n}$ over a ring is \textit{rotationally symmetric} if $a_{i,j}=a_{n+1-i,n+1-j}$ for any $1\le i,j \le n$. An eigenvector for a simple eigenvalue of a rotationally symmetric matrix has the following interesting property; that is, the eigenvector is almost symmetric.

\begin{lemma}\label{sym4}
Let $\bm{A}=(a_{i,j})_{1\le i,j \le n}$ be a rotationally symmetric matrix over $\mathbb{C}$, $\lambda$ be a simple eigenvalue of $\bm{A}$, and $\bm{x}={}^t (x_1, x_2,\dots,x_n)$ is an eigenvector for $\lambda$. Then, we have $x_i= \delta x_{n+1-i}$ for any $1\le i \le n$, where $\delta=\pm 1$.

Moreover, if the sum $\sum_i x_i$ is nonzero, then $\delta =1$.
\begin{proof}
It is easy to show that the vector $\bm{x}'={}^t (x_n,x_{n-1}, \dots,x_1)$ is also an eigenvector for $\lambda$. Since $\lambda$ is simple, the eigenspace for $\lambda$ is a one-dimensional space; therefore, there exists a complex number $\delta$ such that $\delta \bm{x}=\bm{x}'$. Since $\bm{x}$ is an eigenvector, $x_i$ is nonzero for some $i$. Furthermore, since $\delta x_i=x_{n+1-i}$ and $\delta x_{n+1-i} =x_i$, the constant $\delta$ must be $\pm 1$. Hence, the first claim follows.

Moreover, the equality $\delta \bm{x}=\bm{x}'$ implies $\delta\sum_i x_i=\sum_i x_i$. Since $\sum_i x_i$ is nonzero, the second claim follows.
\end{proof}
\end{lemma}

\begin{prop}\label{sym5}
For $d \ge 0$, the $H$-polynomial is self-reciprocal; that is, $H_{i,d} = H_{d+1-i,d}$ for any $0\le i \le d$.
\begin{proof}
By Lemma \ref{sym1} and \ref{sym3}, the matrices $\bm{F}_d$ and $\bm{H}_d$ have the same eigenvalues $0!, 1!,\dots, (d+1)!$ and the vector $$  \bm{T}_d {}^t (F_{-1,d}, F_{0,d}, \dots, F_{d,d}) =  {}^t ( H_{0,d} ,H_{1,d} ,\dots,  H_{d+1,d})$$ is an eigenvector for $(d+1)!$ of $\bm{H}_d$. Lemma 2 (ii) of \cite{BW08} implies that $\bm{H}_d$ is rotationally symmetric, and we have $$ \sum^{d+1}_{i=0}  H_{i,d} = H_d(1) =F_d(0) =1 \not =0.$$ Hence, Lemma \ref{sym4} completes this proof.
\end{proof}
\end{prop}

\begin{rema}
The results in this section give the answer to Problem 1 of \cite{BW08}.
\end{rema}

\subsection{Proof of Main Theorem}

We give a proof of our main theorem. 

We denote $\# \overline{N_i}  (P^{(k)})$ by $\overline{N_i}^{(k)}$. In particular, $\# \overline{N_i} (P)$ is denoted by $\overline{N_i}$.
\begin{proof}[Proof of Theorem \ref{main}]

We have the following recurrence: $$ \bm{F}_d'  \begin{pmatrix}  \overline{N_0}^{(k)} \\ \overline{N_1}^{(k)}  \\ \vdots \\ \overline{N_d}^{(k)}  \end{pmatrix} =\begin{pmatrix}  \overline{N_0}^{(k+1)} \\ \overline{N_1}^{(k+1)}  \\ \vdots \\ \overline{N_d}^{(k+1)}  \end{pmatrix} ,$$ where $\bm{F}_d' =( f_{i,j} )_{0 \le i,j \le d}.$ Consider the generating function $$ M_i(x) = \sum^{\infty}_{k=0} \overline{N_i}^{(k)} x^k$$ for $0\le i \le d.$ We show that \begin{align} M_i(x) &= \sum^{d-i}_{j=0}  \frac{  C_{j,i}  }{ 1-(d+1-j)! x   }  \label{gen1}\end{align} for some rational numbers $C_{j,i}$ and,  in particular, \begin{align} C_{0,i} &= \overline{N_d} F_{i,d}  \label{gen2}\end{align} by descending induction on $i$.

When $i=d$, multiply the recurrence $$f_{d,d} \overline{N_d}^{(k)} = (d+1)! \overline{N_d}^{(k)}  =\overline{N_d}^{(k+1)}$$ by $x^k$ and sum over $k \ge 0$, and we have
\begin{align*}
(d+1)! M_d(x)&= \frac{  M_d(x) -\overline{N_d} }{ x  } \\
M_d&= \frac{ \overline{N_d}  }{  1-(d+1)!x   }.
\end{align*}
Since $F_{d,d}=1$, the claim follows.

Suppose that the claim is true for $i+1, i+2,\dots, d$. By multiplying the recurrence $$ \overline{N_i}^{(k+1)} = f_{i,i} \overline{N_i}^{(k)}   +f_{i,i+1}  \overline{N_{i+1}}^{(k)} +\dots +f_{i,d} \overline{N_d}^{(k)} $$ by $x^k$ and sum over $k \ge 0$, we have \begin{align*} M_i(x)  = \frac{ \overline{N_i}  }{ 1-(i+1)! x   }   +  \sum^d_{j=i+1} \sum^{d-j}_{\ell =0}  \frac{  f_{i,j} C_{\ell, j} x  }{  (1-(i+1)!x) (1-(d+1-\ell )!x)  } .\end{align*} By partial fraction decomposition, the first claim follows, and we have \begin{align*}  & \sum^d_{j=i+1}  \frac{  f_{i,j}  }{  1-(i+1)!x  }   \frac{  C_{0,j} x  }{ 1-(d+1)!x   }  \\  =&\sum^d_{j=i+1}  \frac{ f_{i,j}  C_{0,j}  }{  (d+1)!-(i+1)!    }  \left(  \frac{-1}{1-(i+1)!x)} + \frac{1}{  1-(d+1)!x  } \right) .\end{align*} By the assumption of induction and \eqref{ind2}, we have \begin{align*}  C_{0,i}&= \sum^d_{j=i+1}  \frac{f_{i,j} C_{0,j} }{   (d+1)!-(i+1)! }   \\ &=\frac{   \overline{N_d} }{ (d+1)!-(i+1)!   } \sum^d_{j=i+1} f_{i,j}  F_{j,d } \\ &=\overline{N_d} F_{i,d},\end{align*} and the claim follows.

By \eqref{gen1}, we have \begin{align*} g_{P^{(k)}}(s)&=\sum^d_{i=0} \overline{N_i}^{(k)} s^i (1-s)^{d-i} \\&=\sum^d_{i=0} \left(  \sum^{d-i}_{j=0}  \left(  (d+1-j)! \right)^k C_{j,i} \right) s^i(1-s)^{d-i} .\end{align*} Put $$ \varepsilon_k(s)   =\sum^d_{i=0} \left(  \sum^{d-i}_{j=1}  \left(  (d+1-j)! \right)^k C_{j,i} \right) s^i(1-s)^{d-i} .$$ Then, by \eqref{gen2}, we have \begin{align}  g_{P^{(k)}} (s)&= ((d+1)!)^k \overline{N_d} \sum^d_{i=0}  F_{i,d} s^i (1-s)^{d-i} +  \varepsilon_k(s)   \notag  \\   &=((d+1)!)^k \overline{N_d} H_d(s)  +\varepsilon_k(s) .\label{gen4}\end{align}

Now, we are ready to obtain the result.

By Proposition \ref{sym5}, the coefficient $H_{0,d}$ of $s^d$ of $H_d(s)$ is zero and $H_{1,d}=H_{d,d}$. Since $H_{d,d}=\frac{1}{d!} F_d^{(d)}(-1)=F_{0,d}$ and $F_{0,d}$ is nonzero by the definition, we have the following: $$H_d(s) =F_{0,d} \prod_n (s-\alpha_n)^{e_n} ,$$ where $\sum_n e_n= d-1$ and $\prod_n \alpha_n = (-1)^{d-1}$.

Suppose that $R>0$ is sufficiently large such that the open ball $U(0: R)$ with the center zero of radius $R$ contains all zeros of $H_d(s)$. If $k$ is sufficiently large, then we have $$ | ((d+1)!)^k \overline{N_d} H_d(s)|  >|\varepsilon_k(s)|$$ on the circle $|s|=R$. Hence, Rouche's theorem implies that $g_{P^{(k)}} (s)$ has $d-1$ zeros in $U(0:R)$. Since the leading coefficient of $g_{P^{(k)}}(s)$ is $(-1)^d \chi(P)$ and $\chi(P)$ is nonzero, one of the $d$ zeros of $g_{P^{(k)}} (s)$ must be in the exterior of $U(0:R)$. Hence, the first result follows. Furthermore, since $ g_{ P^{(k)} } (s)$ is a polynomial with integral coefficients, the complex conjugation of the zero is also that of $ g_{ P^{(k)} } (s)$. Hence, the zero must be real.

Suppose that $\varepsilon >0$ is sufficiently small such that the open ball $U(\alpha_n : \varepsilon)$ does not intersect with $U(\alpha_m : \varepsilon)$ if $n \not=m.$ If $k$ is sufficiently large, then the inequality above holds on the circle   $|s-\alpha_n| =\varepsilon$. Hence, Rouche's theorem implies that $g_{P^{(k)}}  (s)$ has $e_n$ zeros in $U(\alpha_n :\varepsilon)$. Hence, $e_n$ zeros of $g_{P^{(k)}}(s)$ converge to $\alpha_n$ as $k\to \infty$. Hence, the second result follows. Since $\prod_n \alpha_n =(-1)^{d-1}$, the third result follows. \end{proof}

\subsection{  The growth of $| \beta_1^{(k)}|$ }

In this section, we estimate the growth of $| \beta_1^{(k)} |$ as $k\to \infty$. 

For two functions $f(x)$ and $g(x)$, define $f(x)\sim g(x)$ if $$\lim_{x\to \infty}  \frac{  f(x) }{  g(x)  } =1.$$

\begin{prop}\label{ES}
Under the same assumption of Theorem \ref{main}, we have  $$  |\beta_1^{(k)} | \sim \frac{   (d+1)!^k H_{1,d}  \# \overline{N_d}  }{    \chi(P) } $$ as $k\to \infty$.
\begin{proof}
By \eqref{gen4}, we have \begin{align*}  g_{P^{(k)}} (s)&=  \sum^d_{i=0}   \overline{N_i}^{(k)} s^i (1-s)^{d-i}   \notag  \\   &=((d+1)!)^k \overline{N_d} H_d(s)  +\varepsilon_k(s) .\end{align*} The degree is $d$ and the leading coefficient is $(-1)^d \chi (P).$ By Proposition \ref{sym5}, the constant term is $(d+1)!^k \overline{N_d} H_{1,d} + \varepsilon_k (s)_0$, where $\varepsilon_k(s)_0$ is that of $\varepsilon_k(s)$. By observing the constant terms of both side $$ g_{P^{(k)}}(s) = (-1)^d \chi(P)\prod^d_{i=1}  (s-\beta_i^{(k)}) ,$$ we have $$  \beta^{(k)}_1= \frac{  (d+1)!^k \overline{N_d} H_{1,d} +o((d+1)!^k)   }{   \chi(P)\prod^d_{i=2}  \beta_i^{(k)}  } ,$$
and Theorem \ref{main} completes this proof.

\end{proof}
\end{prop}

We estimate $H_{1,d}$.

\begin{lemma}\label{2-power}
For $0 \le j \le d,$ we have $h^{(d)}_{0,j} =2^{d-j}.$
\begin{proof}
We prove the claim by induction on $d$.

If $d=0$, then $h^{(0)}_{0,0}=1.$ Hence, the claim is true.

Suppose that the claim is true for $d-1.$ By Lemma \ref{BW}, we have \begin{align*} h^{(d)}_{0,j} =&\sum^{j-1}_{\ell =-1} h^{(d-1)}_{-1,\ell}  + \sum^{d-1}_{\ell =j} h^{(d-1)}_{0,\ell} \\=&1+\sum^{d-1}_{\ell = j}2^{d-1-\ell} \\=&2^{d-j}. \end{align*} Hence, the result follows.
\end{proof}
\end{lemma}

\begin{lemma}\label{simply}
For $d \ge 1,$ we have \begin{multline*}  h^{(d)}_{0,d} \le h^{(d)}_{0,d-1} \le \dots \le h^{(d)}_{0,-1} \le h^{(d)}_{1,d}\le \dots \le h^{(d)}_{1,-1} \le \dots \le  h^{(d)}_{  \left[  \frac{d-1}{2}\right] , d} \le \dots \le  h^{(d)}_{  \left[  \frac{d-1}{2}\right] , N} , \end{multline*}where $$N= \begin{cases} -1 & \text{if $d$ is even} \\ \left[  \frac{d-1}{2}\right] & \text{if $d$ is odd.} \end{cases}$$
\begin{proof}
When $d=1$, the sequence is $h^{(1)}_{0,1} =1 \le 2 = h^{(1)}_{0,0}.$ Hence, the claim follows.

Assume that the claim is true for $d-1$ and $d$ is even. Then, the assumption of induction and Lemma \ref{BW} imply $h^{(d)}_{i-1,-1}=h^{(d)}_{i,d}$ for any $0\le i \le \left[  \frac{d-1}{2}\right]$ and 
$$  h^{(d)}_{i,j-1} -h^{(d)}_{i,j} =-h^{(d-1)}_{i-1,j-1} +h^{(d-1)}_{i,j-1} \ge -h^{(d-1)}_{i-1,-1}+h^{(d-1)}_{i,d-1} =0 $$ for any $0 \le j \le d$ (Lemma \ref{sym5} is used if $0 \le j \le \left[  \frac{d-1}{2}\right]$ ).

If $d$ is odd, for any $0 \le i <\left[  \frac{d-1}{2}\right]$ and $0\le j \le d,$ Lemma \ref{sym5} implies 
\begin{align*}
h^{(d)}_{i,j-1} -h^{(d)}_{i,j}=&-h^{(d-1)}_{i-1,j-1} + h^{(d-1)}_{i,j-1} \\
=&-h^{(d-1)}_{i-1,j-1} +h^{(d-1)}_{d-1-i-1, d-1-(j-1)-1}\\
=&-h^{(d-1)}_{d'-1,j-1} +h^{(d-1)}_{d'-1, 2d'-1}, 
\end{align*}
where $d=2d'+1.$ Since $j-1 \ge 2d' -j,$ we have $h^{(d)}_{i,j-1} \ge h^{(d)}_{i,j}$. Hence, the result follows.
\end{proof}
\end{lemma}

\begin{lemma}\label{det}
Let $$  M=\begin{pmatrix} -A_1&a_{1,2}&\cdots&a_{1,n} \\a_{2,1} &-A_2&&\vdots \\\vdots&& \ddots &\vdots \\a_{n,1} &\cdots&\cdots&-A_n \end{pmatrix}$$ be an $n\times n$ matrix such that $A_j$ and $a_{i,j}$ are positive real numbers for any $i$ and $j$. Suppose that $\sum^n_{i=1, i\not = j} a_{i,j} < A_j$ for any $j$.
\begin{enumerate}
\item The sign of the determinant of $M$ is $(-1)^n$. 
\item If we give positive real numbers $b_1, b_2,\dots, b_n$ and replace the $j$th column of $M$, $1\ le j \le n,$ by ${}^t (-b_1, -b_2,\dots, -b_n)$, denote the matrix $M_j$, then the sign of the determinant of $M_j$ is also $(-1)^n$.
\end{enumerate}
\begin{proof}
We only give a proof of the first claim since the second can be proved similarly.

We prove it by induction on the size of the matrix.

If $n=1$, then we have $|M|=-A_1.$ Hence, te claim follows.

If we assume the truth of the claim for $n-1$, then we have
\begin{align*}
|M|=&\begin{vmatrix} -A_1 &a_{1,2}&a_{1,3}&\cdots&a_{1,n} \\ 0&-A_2+\frac{a_{2,1}}{A_1} a_{1,2}  & a_{2,3}+\frac{a_{2,1}}{A_1} a_{1,3}&\cdots&a_{2,n}+\frac{a_{2,1}}{A_1} a_{1,n} \\
0&a_{3,2}+\frac{a_{3,1}}{A_1} a_{1,2}&-A_3+\frac{a_{3,1}}{A_1} a_{1,3}&\cdots&a_{3,n}+\frac{a_{3,1}}{A_1} a_{1,n} \\
\vdots&\vdots&\vdots&\ddots&\vdots \\
0&a_{n,2}+ \frac{ a_{n,1} }{A_1}  a_{1,2}&a_{n,3} +\frac{ a_{n,1} }{A_1} a_{1,3}&\cdots&-A_n+\frac{ a_{n,1} }{A_1} a_{1,n}
\end{vmatrix} \\
=&-A_1\begin{vmatrix} a_{1,2}&a_{1,3}&\cdots&a_{1,n} \\ 
-A_2+\frac{a_{2,1}}{A_1} a_{1,2}  & a_{2,3}+\frac{a_{2,1}}{A_1} a_{1,3}&\cdots&a_{2,n}+\frac{a_{2,1}}{A_1} a_{1,n} \\
a_{3,2}+\frac{a_{3,1}}{A_1} a_{1,2}&-A_3+\frac{a_{3,1}}{A_1} a_{1,3}&\cdots&a_{3,n}+\frac{a_{3,1}}{A_1} a_{1,n} \\
\vdots&\vdots&\ddots&\vdots \\
a_{n,2}+ \frac{ a_{n,1} }{A_1}  a_{1,2}&a_{n,3} +\frac{ a_{n,1} }{A_1} a_{1,3}&\cdots&-A_n+\frac{ a_{n,1} }{A_1} a_{1,n}.
\end{vmatrix} 
\end{align*}
For any $2\le j\le n,$ the $j$th diagonal entry is negative and the sum of the $j$th column is 
$$-A_j+\sum^n_{i=2,i\not= j} a_{i,j} +\frac{a_{1,j}}{A_1}  \sum^n_{i=2} a_{i,1} <-A_j + \sum^n_{i=1, i\not = j} a_{a,j} <0. $$
Hence, the assumption of induction implies $\mathrm{sign} |M|= -1\times (-1)^{n-1} =(-1)^n,$ and the result follows.

\end{proof}
\end{lemma}

\begin{lemma}\label{positive}
For any $1 \le i \le d,$ the number $H_{i,d}$ is positive.
\begin{proof}
Since the column vector $\mathbb{H} ={}^t ( H_{0,d}, H_{1,d},\dots, H_{d+1,d})$ is an eigenvector for $(d+1)!$ of $\bm{H}_d$, we have $$ \left(   \bm{H}_d -(d+1)!E\right) \begin{pmatrix}  H_{0,d} \\\vdots\\H_{d+1,d}\end{pmatrix} =\begin{pmatrix}  0\\ \vdots \\0 \end{pmatrix}.$$ Since all the entries in the $-1$ and $d$th rows are zero except for $h^{(d)}_{-1,-1} =h^{(d)}_{d,d} =1,$ we have $H_{0,d} =H_{d+1,d} =0.$ Let $\bm{H}_d'=( h^{(d)}_{i,j} )_{    0\le i,j \le d-1}.$ Since the eigenvalue $(d+1)!$ is simple, the rank of $\bm{H}_d'-(d+1)! E$ is $d-1$. The sum of any column in the matrix is zero; therefore, the $(d-1)$th row can be removed. Let $\bm{H}_d'' =(h^{(d)}_{i,j})_{0\le i \le d-2, 0\le j \le d-1}$. We have $$  \left( \bm{H}_d'' -(d+1)! E \right) \begin{pmatrix}  H_{1,d} \\ \vdots \\ H_{d-1,d}\end{pmatrix} =-H_{d,d} \begin{pmatrix}  h^{(d)}_{0,d-1}  \\\vdots \\ h^{(d)}_{d-2,d-1}\end{pmatrix} .$$
If we regard $H_{1,d},\dots, H_{d-1,d}$ as variables and $H_{d,d}$ as a constant, this equation uniquely determines $H_{1,d}, \dots, H_{d-1,d}$. Since $H_{d,d} =F_{0,d}>0$, by Cramer's fromula and Lemma \ref{det}, all $H_{i,d}$ are positive, and the result follows.
\end{proof}
\end{lemma}

\begin{prop}\label{N3}
For $d\ge 1$, we have $$ \frac{ \sqrt{2}^d  }{   (d+1)!d }  \le H_{1,d} \le \frac{    2^{d+1}   }{   (d+1)!   } .$$
\begin{proof}
Since ${}^t (H_{0,d} ,\dots, H_{d+1,d})$ is an eigenvector for $(d+1!)$ of $\bm{H}_d$, we have 
\begin{align*}
h^{(d)}_{0,0} H_{1,d}+h^{(d)}_{0,1}H_{2,d}\dots +h^{(d)}_{0,d-1} H_{d,d}=&(d+1)! H_{1,d} \\
h^{(d)}_{1,0} H_{1,d}+h^{(d)}_{1,1}H_{2,d}\dots +h^{(d)}_{1,d-1} H_{d,d}=&(d+1)! H_{2,d} \\
\vdots& \\
h^{(d)}_{d-1,0} H_{1,d}+h^{(d)}_{d-1,1}H_{2,d}\dots +h^{(d)}_{0,d-1} H_{d,d}=&(d+1)! H_{d,d}.
\end{align*}
Since $\sum^d_{i=1} H_{i,d}=1$ and all $H_{i,d}$ are positive (Lemma \ref{positive}), Lemma \ref{2-power} implies \begin{align*} (d+1)!H_{1,d} =& h^{(d)}_{0,0} H_{1,d} +\dots + h^{(d)}_{0,d-1} H_{d,d} \\\le& h^{(d)}_{0,0}  +\dots +h^{(d)}_{0,d-1} \\\le & 2^{d-1}.\end{align*}
 Hence, the right inequality follows.
 
By Lemma \ref{simply}, \ref{positive}, and Proposition \ref{sym5}, $H_{   \left[  \frac{d+1}{2}\right],d   }$ is the maximum among $H_{1,d},\dots, H_{d,d}$. Since $\sum^d_{i=1} H_{i,d}=1,$ we have $H_{ \left[  \frac{d+1}{2}\right],d} \ge \frac{1}{d}.$ Hence, Lemma \ref{2-power} implies $$ (d+1)!  H_{1,d}> h^{(d)}_{0,\left[  \frac{d+1}{2}\right]-1} H_{ \left[  \frac{d+1}{2}\right],d  }   \ge 2^{  d-\left[  \frac{d+1}{2}\right]+1  }  \frac{1}{d}  \ge \sqrt{2}^d \frac{1}{d}.$$
Hence, the result follows.

\end{proof}
\end{prop}

\section{A plan for application}
We give a plan for application of the main theorem to the distribution of primes.

\subsection{Historical background}

For a positive real number $x$, let $\pi(x)$ be the number of primes not exceeding $x$. This function is in a central place in number theory, and it irregularly increases as $x\to \infty$. However, the prime number theorem states that, surprisingly, the elementary function $\frac{x}{\log x}$ approximates to $\pi(x)$: $$ \lim_{x\to \infty}  \frac{\pi(x)}{  x/\log x   } =1.$$ The first proof of the fact was given by Hadamard and de la Vall\'ee Poussin independently, and they used the Riemann zeta function $\zeta(s)$. See, for example, Chapter III of \cite{Tit86}.

The prime number theorem has many equivalent propositions, and Bj\"orner gave a topological interpretation to one of them \cite{Bjo11}.

For a squarefree positive integer $k$, let $P(k)$ be the set of prime factors of $k$. For any $n\ge 2$, define an abstract simplicial complex $\Delta_n$ to be the set of $P(k)$ for all squarefree integers $2\le k\le n.$ A family $\Delta$ of nonempty subsets of a finite set is an \textit{abstract simplicial complex} if $\Delta$ is closed under taking subsets. Then, he gave the following equivalence: $$ \text{Prime Number Theorem} \Leftrightarrow\chi(\Delta_n) =o(n) .$$ Moreover, he pointed out the following equivalence: $$  \text{Riemann Hypothesis} \Leftrightarrow \chi(\Delta_n) =O(n^{\frac{1}{2} +\varepsilon}).$$The Euler characteristic $\chi(\Delta)$ of an abstract simplicial complex $\Delta$ is given by $\chi(\Delta)=\sum_{A \in \Delta} (-1)^{\# A -1}$. In fact, $\chi(\Delta_n)$ is almost the Mertens function $M(n)$; that is, $-M(n)=\chi(\Delta_n)-1$. The function $M(x)$ is defined by $\sum_{n\le x} \mu (n)$, where $\mu(n)$ is the classical M\"obius function. If $n=p_1 p_2\dots p_k$, $p_i\not = p_j$, then $\mu(n)=(-1)^k$ and $(-1)^{\# P(k)  -1  }  =(-1)^{k-1}$. Since $\mu(1)=1$, the equality follows. Theorem 4.14 and 4.15 of \cite{Apo} and Theorem 14.25 (C) imply the two equivalences.

The first step to study $\chi(\Delta_n)$ should be to study the homology group of $\Delta_n$. Bj\"orner tried it, but he found that $\Delta_n$ has the homotopy type of a wedge of spheres. Namely, the homology group of $\Delta_n$ is almost trivial. He concluded that ``perhaps a study of deeper topological invariants of $\Delta_n$ could add something of value''.

\subsection{Strategy}
We give a continuation of Bj\"orner's work by the zeta series of finite posets.

\begin{defi}
Define a poset $P_n$ , $n\ge 2,$ to be the set of squarefree integers $2\le k \le n$ and give an order by divisibility. Namely, $a\le b$ if and only if $a| b$.
\end{defi}

The dimension of $P_n$ is $d$ if and only if $p_1p_2\dots p_{d+1} \le n < p_1p_2\dots p_{d+2}$, where $(p_1, p_2,p_3,\dots)$ is the sequence of primes $(2,3,5,\dots)$. 

\begin{exam}
The poset $P_6$ is $$  
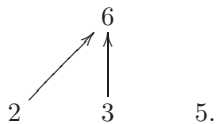
\xymatrix{ &6&& \\  2\ar[ur] &3\ar[u]&5 .}$$ The Euler characteristic is two and the dimension is one.

The poset $P_{30}$ is $$
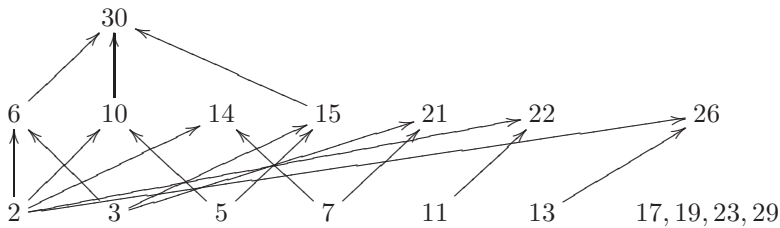
\xymatrix{   &30&&&&& \\6\ar[ur]&10 \ar[u]&14&15 \ar[ull]&21&22&26 \\
2\ar[u] \ar[ur] \ar[urr] \ar[urrrrr] \ar[urrrrrr] &3 \ar[ul] \ar[urr] \ar[urrr] &5 \ar[ur] \ar[ul] &7 \ar[ul] \ar[ur] &11\ar[ur] &13\ar[ur] &17,19,23,29}$$ The Euler characteristic is four and the dimension is two.
\end{exam}

We have the increasing sequence of posets:
$$ *=P_2\subset P_3=P_4 \subset P_5\subset P_6\subset \cdots. $$

\begin{exam}
If $n$ is small, it is easy to compute $\chi(P_n)$.

\begin{center}
\begin{tabular}{ccccccccccccccc} \toprule
$n$&2&3&4&5&6&7&8&9&10  &11&12&13&14&15 \\ \midrule
$\chi(P_n)$&1&2&2&3&2&3&3&3&2 &3&3&4&3&2 \\ \bottomrule
\end{tabular}

\begin{tabular}{ccccccccccccccc} \toprule
$n$&16&17&18&19&20&21&22&23&24&25&26&27&28&29 \\ \midrule
$\chi(P_n)$ &2&3&3&4&4&3&2&3&3&3&2&2&2&3  \\ \bottomrule
\end{tabular}

\begin{tabular}{ccccccccccccccc} \toprule
$n$&30&31&32&33&34&35&36&37&38&39&40&41&42&43  \\ \midrule
$\chi(P_n)$&4&5&5&4&3&2&2&3&2&1&1&2&3&4 \\ \bottomrule

\end{tabular}

\end{center}

\end{exam}

We can observe the oscillation of  $\chi(P_n) $. The Euler characteristic $\chi(P_n)$ is not always positive. Indeed, $\chi(P_{95})=-1$ and 95 is the smallest integer whose Euler characteristic is negative.

The poset $P_n$ is the face poset of $\Delta_n$. For an abstract simplicial complex $\Delta$, the \textit{face poset} $F(\Delta)$ of $\Delta$ is $\Delta$ itself as a set and its order is given by inclusion. The Euler characteristics of $\Delta$ and $F(\Delta)$ coincide. Indeed, we have
\begin{align*}
\chi(F(\Delta))=&\sum^{\dim F(\Delta)}_{i=0} (-1)^i \# \overline{N_i} (F(\Delta)) \\
=&\sum_{A\in \Delta} \sum^{\dim \Delta}_{i=0} (-1)^i \# \overline{N_i} (F(\Delta))_A \\
=&\sum^{\dim \Delta}_{d=0} \sum_{A\in \Delta, \# A=d+1} \sum^d_{i=0} (-1)^i f_{i,d} \\
=&\sum^{\dim \Delta}_{d=0} \sum_{A\in \Delta, \# A=d+1} (-1)^d \\
=&\chi(\Delta),
\end{align*}
where $\overline{N_i}(F(\Delta))_A$ is the set of chains of length $i$ in $F(\Delta)$ whose target is $A$, and $f_{i,d}$ is the number defined in \S \ref{vari}. It is easy to show $$ \sum^d_{i=0} (-1)^i f_{i,d} =(-1)^d $$ by \eqref{ind1} and induction on $d$.

We obtain the following: 
$$ \text{Prime Number Theorem} \Leftrightarrow\chi(P_n) =o(n) .$$$$  \text{Riemann Hypothesis} \Leftrightarrow \chi(P_n) =O(n^{\frac{1}{2} +\varepsilon}).$$
Hence, it is important to estimate $|\chi(P_n)|$ as $n\to \infty$.

In addition, the following estimation $$  \int^X_2 \left(  \frac{\chi(P_{[x]})}{x}  \right)^2 dx = O(\log X) $$ implies the simplicity of the zeros of the Riemann zeta function by Theorem 14.29(A) of \cite{Tit86}. The simplicity and the Riemann Hypothesis are major problems in number theory.

At first glance, $|\chi(P_n)|$ is very smaller than $n$ in the table, however, the oscillation is very complicated if $n$ is large. Although Mertens conjectured $$ |M(x)| \le \sqrt{x} $$ for $x>1$, Odlyzko and Riele disproved it \cite{OR85}. Namely, the inequality is violated infinitely many times. They showed the existence of counter examples to the conjecture, but no examples have been found concretely.

Let us begin to try the problem.

Assume that $\chi(P_n)$ is nonzero (we do not need to estimate $\chi(P_n)$ if it is zero). By Lemma \ref{zero}, we have $$  \chi(P_n) = \frac{  \# \overline{N_0}(P_n)  }{ \prod_{i=1}^{d_n} \beta_{i,n}  } ,$$where $d_n =\dim P_n.$ We have to estimate all $| \beta_{i,n}|$ as precisely as possible, but it is getting very harder since the number of zeros $d_n$ steadily increases as $n \to \infty$. By applying the main theorem, we obtain the following:

\begin{lemma}\label{N5}
If $n \ge 6$ and $\chi(P_n)\not = 0$, then $$ |\chi(P_n)|  \sim \frac{   (d_n+1)!^{k_n} H_{1,d_n}  \# \overline{N_{d_n}} (P_n)  }{       |\beta^{(k_n)}_{1,n}| } $$ as $k\to \infty.$
\begin{proof}
It directly follows from Proposition \ref{ES}.
\end{proof}
\end{lemma}

Hence, $|\chi(P_n)|$ is almost the right hand side. The product of the zeros except for $\beta^{(k_n)}_{1,n}$ converges to $(-1)^{d_n}$. Only one of the zeros remains, and we do not have any loss at this stage.

Next, we estimate the growth of the dimension of $P_n$.

\begin{prop}\label{dim}
We have $$ d_n=\dim P_n= \frac{\log n}{ \log \log n}   +O\left(  \frac{\log n}{(\log \log n)^2}\right) .$$
\begin{proof}
By Theorem 4.7 of \cite{Apo}, we have $$ C_1 n\log n < p_n < C_2 n\log n  $$ for some constants $C_1, C_2,$ and any $n\ge 2$. For example, put $C_1=\frac{1}{6}$ and $C_2=24$. Then, $C_1< p_1=2<C_2$. By the definition, $\dim P_n =d$ if and only if $$ p_1 p_2\dots p_{d+1} \le n < p_1 p_2 \dots p_{d+2}.$$ Hence, we have 
\begin{align}
C_1^{d+1} (d+1)! \prod^{d+1}_{m=2} \log m < n < C_2^{d+2} (d+2)! \prod^{d+2}_{m=2} \log m. \label{put1}
\end{align}
Since the function $\frac{\log x}{\log \log x}$ is simply increasing in $[e^e,\infty)$, we have 
$$ \frac{\log n}{\log \log n} < \frac{   (d+2) \log C_2 +\log (d+2)!  +\sum^{d+2}_{m=2} \log \log m     }{  \log \left( (d+2) \log C_2 +\log (d+2)!  +\sum^{d+2}_{m=2} \log \log m   \right)   }$$ if $n\ge e^e$. By integral test, we have
\begin{multline*}
d \log \log d - \Li d+C\le \sum^{d+1}_{m=2} \log \log m \le (d+1) \log \log (d+1) - \Li (d+1) +C,
\end{multline*}
where $\Li x=\int^x_2 \frac{dt}{\log t}$ and $C= \int^e_2 \frac{dt}{\log t} + \log \log 2$. By Stirling's formula, we have , for $\varepsilon >0$, 
\begin{align}
\frac{\log n}{\log \log n }<& \frac{   (d+2)\log C_2 + (d+3)\log (d+3) + (d+2)\log \log (d+2)      }{    \log (d+3) +\log \log (d+3)     } \notag \\
<& (1+\varepsilon) d\label{put2}
\end{align}
if $n$ is sufficiently large. If we replace $\varepsilon$ by$$\frac{2\log C_2}{   \log d +\log \log d} ,$$ the inequality \eqref{put2} holds. Hence, we have
\begin{align*}
d>&\frac{1}{1+\varepsilon} \frac{\log n}{\log \log n}\\
=& \frac{\log n}{\log \log n} -\frac{\varepsilon}{1+\varepsilon}\frac{\log n}{\log \log n} \\
>&\frac{\log n}{\log \log n}-\frac{2\log C_2}{\log d} \frac{\log n}{\log \log n}.
\end{align*}
If we put $\varepsilon =\frac{1}{2}$ and take logarithm in \eqref{put2}, we have
\begin{align}
\log d > \log \left(   \frac{2}{3}  \frac{\log n}{\log \log n} \right)> \frac{1}{2} \log \log n. \label{put3}
\end{align}
Hence, we obtain $$ d> \frac{\log n}{\log \log n} -4\log C_2  \frac{\log n}{  (\log \log n)^2   }  .$$

By the left inequality of \eqref{put1}, we can similarly obtain $$ \frac{\log n}{\log \log n} \ge  (1-\varepsilon)d, $$ and we can replace $\varepsilon $ by $$  \frac{2\log C_1}{\log d +\log \log d} . $$By \eqref{put3}, we obtain $$  d\le \frac{\log n}{\log \log n}   +4 \log C_1 \frac{\log n}{ ( \log \log n )^2  }.$$ Hence, the result follows.
\end{proof}
\end{prop}

If $n=p_1^{e_1} p_2^{e_2}\dots p_d^{e_d}$, then the \textit{weight} of $n$ is defined by $\sum^d_{i=1} e_i$. Denote $\pi_d(x)$ the number of positive squarefree integers of weight $d$ not exceeding a real number $x$.

\begin{lemma}\label{N1}
We have $$  \pi_d(p_1p_2\dots p_{d+1} )  \le C^d $$ for some constant $C>0$.
\begin{proof}
Suppose that a sequence of primes $q_1, q_2,\dots, q_d$ satisfies $q_1 q_2\dots q_d \le p_1 p_2 \dots p_{d+1}$ and $q_1<\dots <q_d$. We count the number of such $(q_1,\dots, q_d)$. The greatest member $q_d$ does not exceed $p_{d^3}$, since $q_1 q_2\dots q_d \ge p_1p_2 \dots p_{d-1} q_d$; that is, $$    q_d \le p_d p_{d+1}  \le C_2^2 (d+1)^2 \log^2 (d+1) . $$ Hence, we choose $q_1,q_2,\dots, q_d$ from the set $\{  p_1,p_2,\dots,p_{d^3}    \} .$

We choose $m$ satisfying $m \ge \frac{e^2 C_2}{C_1}$.  The number of $i, 1\le i \le d,$ such that $q_i \ge p_{md}$ is smaller than $[\log d]$. Indeed, if $ q_1,q_2,\dots, q_{d-[\log d]}   <p_{md} $ and $ q_{d-[\log d] +1} , \dots, q_d \ge p_{md}$, then we have 
$$  q_1q_2 \dots q_d \ge p_1 p_2 \dots p_{d-[\log d]}  p_{md}p_{md+1}  \dots p_{md+[\log d]} .$$ Hence, it must be that $$p_{md}\dots p_{md+[\log d]} \le p_{d-[\log d] +1}  \dots p_{d+1} .$$ The left hand side is, at least, $$  p_{md} \dots   p_{md+[\log d]} \ge \left(  C_1 md \log d  \right)^{[\log d]} ,$$ and the right hand side is, at most, $$ p_{d-[\log d] +1} \dots p_{d+1} \le \left(  C_2 d \log d  \right)^{[\log d]+1}  . $$ However, the condition $m\ge \frac{e^2 C_2}{C_1}$ implies $$ p_{md} \dots p_{md+[\log d]} > p_{d-[\log d]+1} \dots p_{d+1} .$$ Hence, the claim follows.

By Stirling's formula, we have 
\begin{align*}
\pi_d (p_1\dots p_{d+1}) \ll &\binom{md}{d} \frac{d^{3[\log d]}}{[\log d]!} \\
\ll & \frac{m^d d^d}{d!} d^{3\log d} \\
\ll & \frac{e^d m^d d^d}{d^d}d^{3\log d}\\
\ll &C^d
\end{align*}
for some constant $C>0$. Hence, the result follows.
\end{proof}
\end{lemma}

\begin{prop}\label{N2}
We have $$ \# \overline{N_{d_n}} (P_n)  =O\left(  (d_n+1)!  C^{d_n} \right) $$ for some constant $C$.
\begin{proof}
Since $n < p_1 p_2 \dots p_{d_n +2}$, we have 
\begin{align*}
\# \overline{N_{d_n}} (P_n) \le& \# \overline{N_{d_n}} (P_{p_1p_2\dots p_{d_n +2}})  \\
=&(d_n+1)! \pi_{d_n+1} (p_1p_2 \dots p_{d_n+2}),
\end{align*}
and Proposition \ref{N1} completes this proof.
\end{proof}

\end{prop}

The last problem is to estimate $| \beta_{1,n}^{(k_n)} |$.

\begin{defi}
By Proposition \ref{ES}, we can put $$ | \beta_{1,n}^{(k_n)} |\sim \alpha_n(d_n+1)!^{k_n} $$ as $k_n \to \infty$ if $\chi(P_n)\not = 0$, since $\chi(P_n), H_{1,d_n}$, and $\# \overline{N_{d_n} }(P_n)$ are constant for any fixed $n\ge 6$. Namely, define $$  \alpha_n =\frac{  H_{1,d_n}  \# \overline{N_{d_n}} (P_n) }{    \chi(P_n )}$$
\end{defi}

\begin{exam}
We compute $\alpha_n$ when $n$ is small. 

\begin{center}

\begin{tabular}{ccccccccccccccc} \toprule
$n$&6&7&10&11&13&14&15&17&19&21&22&23&26&29 \\ \midrule
$\alpha_n$&1& $\frac{2}{3}$&2&$\frac{4}{5}$&1&2&4&$\frac{8}{3}$&2&$\frac{10}{3}$&6&4&7&$\frac{14}{3}$  \\ \bottomrule 
\end{tabular}

\begin{tabular}{c|cccccccccccccc} \toprule
$n$&30&31&33&34&35&37&38&39&41&42&43&46&47&51     \\ \midrule
$\alpha_n$&$\frac{3}{4}$&$\frac{3}{5}$  &$\frac{3}{4}$ &1&$\frac{3}{2}$  &1&$\frac{3}{2}$&3&$\frac{3}{2}$ &2&$\frac{3}{2}$ &2&$\frac{3}{2}$ &2   \\ \bottomrule 
\end{tabular}

\begin{tabular}{ccccccccccccccc} \toprule
$n$&53&55&57&58&59&61&62&65&66&67&69&70&71&73 \\ \midrule
$\alpha_n$&$\frac{3}{2}$ &2&3&6&3&2&3&6&$\frac{9}{2}$ &3&$\frac{9}{2}$ &4&3&$\frac{12}{5}$  \\ \bottomrule 
\end{tabular}

\begin{tabular}{cccccccccc|ccc} \toprule
$n$&  74&$\cdots$&199&201&202&203&205&206&209&210&211&213  \\ \midrule
$\alpha_n$&3&$\cdots$&    $\frac{19}{3}$ &$\frac{57}{8}$ &$\frac{57}{7}$ &$\frac{19}{2}$&$\frac{57}{5}$ &$\frac{57}{4}$&19&$\frac{24}{11}$ &$\frac{16}{11}$ &$\frac{24}{11}$    \\ \bottomrule 
\end{tabular}

\end{center}

When the dimension increases, $\alpha_n$ suddenly decreases since $H_{1,d_n}$ and $\# \overline{N_{d_n} } (P_n)$ do so.

\end{exam}

The hardest part is to estimate $\alpha_n$. Due to the part, we do not complete this application.

We introduce the following conjecture and its consequence:

\begin{conj}
We have $$ \alpha_n \gg \frac{1}{(d_n+1)!} $$ or $$  \alpha_n  \gg_{\varepsilon} \frac{1}{(d_n+1)!^{\frac{1}{2} +\varepsilon}} $$
for any $\varepsilon >0$.
\end{conj}

I think that we should estimate $\alpha_n$ by some function of $d_n$ rather than that of $n$ since we do $H_{a,d_n}$ and $\# \overline{N_{d_n}} (P_n)$ so.

\begin{prop}\label{con1}
If $\alpha_n \gg \frac{1}{(d_n+1)!}$, then we have $$   \chi(P_n) =O\left(  n \exp \left(  -A\frac{\log n \log \log \log n}{\log \log n} \right)  \right)  $$ for some constant $A>0$.
\begin{proof}
Proposition \ref{ES}, \ref{N3}, and Corollary \ref{N2} imply
\begin{align*}
\chi(P_n)  \ll& \frac{(d_n+1)!^{k_n}}{\alpha_n (d_n+1)!^{k_n}} \frac{2^{d_n}}{(d_n+1)!}(d_n+1)!C^{d_n} \\
\ll& \frac{C^{d_n}}{\alpha_n} \\
\ll& (d_n+1)!C^{d_n} \\
\ll& d_n^{d_n}C^{d_n} \\
=&\exp\left(  d_n \log d_n  +d_n \log C  \right).
\end{align*}
By Proposition \ref{dim}, we have
\begin{align*}
\chi(P_n)  \ll &  \exp \left(  \frac{\log n}{\log \log n} \log \left(  \frac{\log n}{\log \log n} \right) +A\frac{\log n}{\log \log n}   \right) \\
=&\exp \left(  \log n -\frac{\log n}{\log \log n} \log \log \log n +A\frac{\log n}{\log \log n}   \right)\\
\ll& n \exp \left(  -A \frac{\log n \log \log \log n}{\log \log n}   \right)
\end{align*}
for some constant $A>0$. Hence, the result follows.
\end{proof}
\end{prop}

The result improves the best result of $M(x)$: $$  M(x)=O \left( x\exp \left( -B \log^{\frac{3}{5}} x (\log \log x)^{-\frac{1}{5}} \right)  \right) . $$
See, for example, Theorem 12.7 of \cite{Ivic}.

Similarly, we obtain the following:

\begin{prop}
If $  \alpha_n  \gg_{\varepsilon} \frac{1}{(d_n+1)!^{\frac{1}{2} +\varepsilon}} $ for any $\varepsilon>0$, the Riemann Hypothesis is true.
\end{prop}

Note that Proposition \ref{dim} is precise and it is meaning less to improve Proposition \ref{N3} and Lemma \ref{N1}. If we supposed $$ H_{1,d} =\frac{  \sqrt{2}  }{  (d+1)! d  } $$ and $$\pi_d(p_1p_2\dots p_{d+1})=1 ,$$ we would have $\chi(P_n) \ll \frac{     \sqrt{2}^{d_n}  }{ \alpha_n  }$, and this is almost $\frac{    C^{d_n} }{ \alpha_n   }$.

At the beginning, our approach is topological, however we only use the two elementary topological notions: Euler characteristic and barycentric subdivision. The proof of Theorem \ref{main} is almost elementary, except for Rouche's Theorem, and Proposition \ref{N3}, \ref{dim}, \ref{N2} are also. We use Theorem 4.7 of \cite{Apo} in this section, however note that the proof is also elementary.

The result of Proposition \ref{con1} is very strong; therefore, I guess that to estimate $\alpha_n$ requires us higher techniques in topology.

\end{document}